\documentclass[12pt,english]{article}

\usepackage[T1]{fontenc}
\usepackage[utf8]{inputenc}
\usepackage[main=english]{babel}

\usepackage{geometry}
\usepackage{amsmath,amsthm,amssymb,amsfonts,mathtools}
\usepackage{mathrsfs}
\usepackage{esint}

\usepackage{graphicx}   % pdflatex-friendly; use PDF/EPS/JPG/TIFF
\usepackage{float}
\usepackage{tikz}       % OK for vector figures (compile with pdflatex)

\usepackage{comment}
\usepackage{indentfirst}

\theoremstyle{definition}

\newtheorem{definition}{Definition}[section]
\newtheorem{proposition}{Proposition}[section]

\title{Generic density of stationary geodesic nets that are not closed geodesics}
\author{Talant Talipov}
\date{}

\begin{document}
\thispagestyle{empty}
\maketitle

\begin{abstract}
    We prove that for a Baire-generic Riemannian metric on a closed
    smooth manifold of dimension greater than or equal to $3$, the union of stationary geodesic nets that are not closed geodesics forms a dense set. This result confirms a Nabutovsky--Parsch conjecture in this case.
\end{abstract}

\section{Introduction}
A classical question in Riemannian geometry, going back to Poincar\'e, asks whether every closed manifold admits infinitely many periodic geodesics. By a result of Lyusternik and Fet~\cite{LyusternikFet} any closed manifold has at least one closed geodesic, and many results establish infinitely many in special cases. For example, Rademacher showed that for a generic metric on any compact simply-connected manifold there are infinitely many closed geodesics~\cite{Rademacher}. Bangert and Franks proved that every Riemannian metric on $S^2$ admits infinitely many closed geodesics~\cite{Bangert,Franks}, and Hingston later gave lower bounds on the number of closed geodesics of length $\le x$ on $S^2$~\cite{Hingston}.

On the other hand, Katok constructed striking examples of irreversible Finsler metrics on rank--one symmetric spaces that have only finitely many closed geodesics~\cite{Katok}. In particular, Katok showed that on spheres, real and complex projective spaces, quaternionic projective spaces, and the Cayley plane, there exist non-reversible Finsler metrics each of which has only finitely many distinct prime closed geodesics (see also Ziller's study of these examples~\cite{Ziller}). These constructions demonstrate that topology alone does not force infinitely many periodic geodesics under arbitrary metrics.

\begin{figure}[H]
    \centering
    % Please crop the image tightly; adjust trim values as needed: trim=left bottom right top
    \includegraphics[width=4.5in,trim=0 0 0 0,clip]{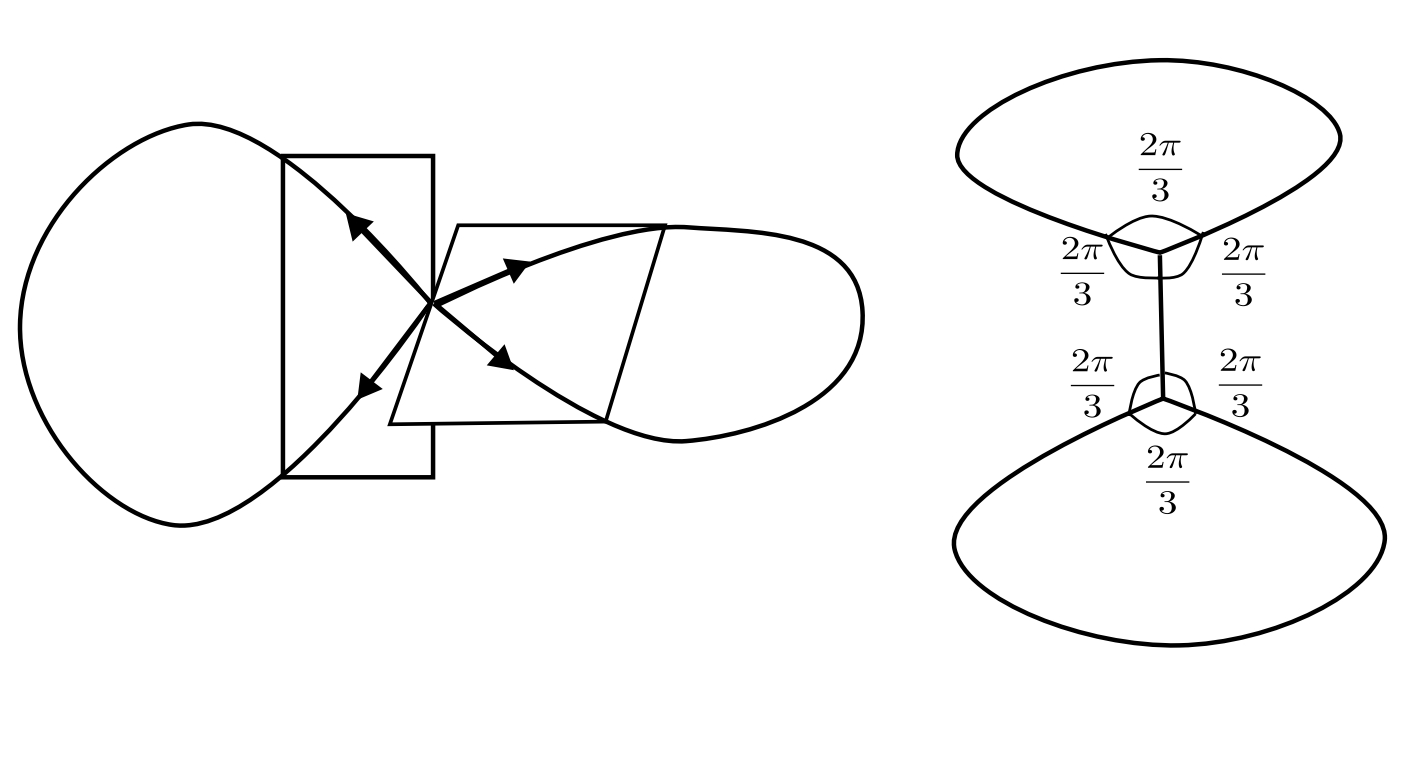}
    \caption{Stationary twisted figure-eight (left) and stationary eyeglass (right)}
    \label{drawing_1}
\end{figure}

A \textit{geodesic net} is a finite weighted graph immersed in $(M,g)$ whose edges are geodesic segments. A geodesic net is called \textit{stationary} if it is a critical point of the length functional $L_g$ with respect to $g$. This is equivalent to the condition that the sum of the inward pointing unit tangent vectors (with multiplicity) is zero at every vertex. Allard and Almgren showed that any one-dimensional stationary integral varifold of positive density is precisely a stationary geodesic net~\cite{AllardAlmgren}. In particular, any union of closed geodesics is a trivial example of a stationary geodesic net.

There are relatively few general existence results for stationary geodesic nets. Nabutovsky and Rotman~\cite{Rotman,NabutovskyRotman} constructed short stationary geodesic nets on any closed manifold. Recently Liokumovich--Staffa proved that for a generic Riemannian metric the union of all stationary geodesic nets is dense in the manifold~\cite{LiokumovichStaffa}, and Li--Staffa further showed that in dimension $3$ this generic density result can be refined: for a generic metric, the embedded stationary geodesic nets are equidistributed throughout $M$~\cite{LiStaffa}. Staffa later extended equidistribution to all dimensions $n \geq 3$~\cite{Staffa_Weyl}. However, these results allow for the possibility that the stationary geodesic nets in question are unions of closed geodesics, i.e., they may be trivial as geodesic nets. See also analogous generic results for minimal hypersurfaces in~\cite{IrieMarquesNeves,MarquesNevesSong}.

On the other hand, there are examples of stationary geodesic nets that contain no closed geodesics. Hass--Morgan~\cite{HassMorgan} proved that any convex metric on $S^2$ which is sufficiently $C^2$--close to the round metric admits a stationary geodesic $\theta$--graph. Cheng~\cite{Cheng} showed that for each $n\ge3$ there exists a closed $n$--manifold $M$ and an open set $U$ in the $C^\infty$--topology such that every metric $g\in U$ admits a stationary geodesic net containing no closed geodesics. Notably, these are the only known results that guarantee the existence of stationary geodesic nets not composed of periodic geodesics on an open set of metrics on a closed manifold.

We call a stationary geodesic net \textit{essential} if it is not a union of closed geodesics. Simple examples of essential stationary geodesic nets in $M^n$ for $n \geq 3$ are illustrated in Fig.~\ref{drawing_1}. Stationary twisted figure-eight consists of one vertex and two geodesic loops based at that point. A stationary eyeglass consists of two vertices connected by a geodesic edge, with a geodesic loop attached at each vertex.

\textbf{Main Theorem.} Let $M^n$ be a closed manifold with $n\ge3$. For a $C^\infty$‑generic Riemannian metric $g$ on $M$, the union of all essential embedded stationary geodesic nets is dense in $M$.

This establishes a stronger version of a Nabutovsky--Parsch conjecture (see Problem~2.0.1 in~\cite{NabutovskyParsch}) for generic metrics on closed manifolds of dimension at least three.

\textbf{Outline.}
The paper is structured as follows. In Section~\ref{sec-2}, we present an overview of geodesic nets. In Section~\ref{sec-3}, we give a proof of Main Theorem. To prove genericity we need to establish both openness and denseness. Openness follows from the Structure Theorem for stationary geodesic nets. To prove denseness we combine the Liokumovich--Staffa density result with a localized metric perturbation that produces an essential stationary geodesic net meeting any prescribed open set. Concretely, given two closed geodesics that lie arbitrarily close we consider two cases. If they are disjoint, we join them by a minimizing geodesic segment and perform a localized conformal deformation so that the inward unit tangent vectors at the joining vertex (with suitable multiplicities) sum to zero, producing an eyeglass geodesic net (see Figure~\ref{drawing_1}, right). If they intersect, we perturb the metric to turn the intersection into a nontrivial singular vertex whose combinatorial model is a twisted figure-eight geodesic net (see Figure~\ref{drawing_1}, left). In both cases the perturbations are supported in small, disjoint neighborhoods and can be made arbitrarily $C^k$-small. This yields density.

\textbf{Acknowledgments.}
The author is grateful to Prof. Yevgeny Liokumovich for his invaluable supervision and encouragement throughout this work, and to Bruno Staffa for generously explaining the Structure Theorem for stationary geodesic nets. Part of this work was carried out during the author's stay at Nazarbayev University in Astana; the author thanks Prof. Durvudkhan Suragan and Makhpal Manarbek for their hospitality. This work was supported by the Dr. Sergiy and Tetyana Kryvoruchko Graduate Scholarship in Mathematics.

\section{Geodesic nets} \label{sec-2}
Let $M$ be an $n$-dimensional smooth manifold. We recall the framework for geodesic nets introduced by Staffa~\cite{Staffa}.
\begin{definition}
    A \textit{weighted multigraph} is a graph $\Gamma = (\mathscr{E} , \mathscr{V}, \{\pi(E)\}_{E \in \mathscr{E}}, \{n(E)\}_{E \in \mathscr{E}})$ consisting of:
    \begin{itemize}
        \item A set of edges $\mathscr{E}$. For each $E \in \mathscr{E}$, we fix a homeomorphism $\imath_E : E \rightarrow [0, 1]$.
        \item A set of vertices $\mathscr{V}$.
        \item For each $E \in \mathscr{E}$, a map $\pi_E : \{0, 1\} \rightarrow \mathscr{V}$ which sends each of the boundary points of the edge $E$ (identified with 0 and 1) to their corresponding vertex $v$.
        \item A multiplicity $n(E) \in \mathbb{N}$ assigned to each edge $E \in \mathscr{E}$.
    \end{itemize}
\end{definition}
\begin{definition}
    A \textit{$\Gamma$-net} $G$ on $M$ is a continuous map $G : \Gamma \rightarrow M$ which is a $C^2$ immersion when restricted to the edges of $\Gamma$.
\end{definition}
\begin{definition}
    We say that a $\Gamma$-net $G$ is \textit{embedded} if the map $G : \Gamma \rightarrow M$ is injective (notice that by the compactness of $\Gamma$ this is equivalent to say that the map $G : \Gamma \rightarrow M$ is a homeomorphism onto its image).
\end{definition}
In the following we will omit the superscript $k$ for simplicity, assuming it is fixed. Given $g \in M$ and $G \in \Omega(\Gamma, M )$, we define the $g$-length of $G$ by
\begin{equation*}
    l_g(G) = \int_\Gamma \sqrt{g_{G(t)}\Bigl(\dot{G}(t), \dot{G}(t)\Bigr)} dt
\end{equation*}
where given a measurable function $h: \Gamma \rightarrow \mathbb{R}$ which is integrable along each edge $E \in \mathscr{E}$, we define
\begin{equation*}
    \int_\Gamma h(t) dt = \sum_{E \in \mathscr{E}} n(E) \int_{[0,1]} h(t)dt.
\end{equation*}
\begin{definition}
    A $\Gamma$-net $G \in \Omega(\Gamma, M)$ is a \textit{stationary geodesic network} with respect to the metric $g \in M$ if it is a critical point of the length functional $l_g : \Omega(\Gamma, M) \rightarrow \mathbb{R}$.
\end{definition}
A constant speed parametrized $\Gamma$-net $G_0$ is stationary with respect to $l_g$ if and only if:
\begin{itemize}
    \item $G_0(t) = 0$ along each edge $E \in \mathscr{E}$ (i.e. the edges of $\Gamma$ are mapped to geodesic segments).
    \item $V(G_0)(v) = 0$ for all $v \in V$ . This means that the sum with multiplicity of the inward unit tangent vectors to the edges concurring at each vertex $v$ must be 0.
\end{itemize}

\section{Proof of Main Theorem} \label{sec-3}
% (The remainder of the paper is unchanged except for minor edits to notation and references.)

\begin{proposition}\label{prop}
    Suppose $n \geq 3$ and let $k \in \mathbb{Z}_{> 0}$ and $U \subset M$ be a nonempty open set. The set $\mathcal{M}^k_U$ of $C^k$-smooth Riemannian metrics on $M^n$ for which there exists a nondegenerate, essential, embedded stationary geodesic net $\gamma$ intersecting $U$ is open and dense in the $C^k$-topology.
\end{proposition}

\begin{proof}[Proof of Proposition \ref{prop}]
    Let $\tilde g \in \mathcal{M}^k_U$ and $\tilde\gamma$ be as in the proposition statement. Since $\tilde\gamma$ is nondegenerate, the Inverse Function Theorem implies that for every Riemannian metric $g$ sufficiently close to $\tilde g$, there exists a nondegenerate, essential, embedded stationary geodesic net $\gamma$ close to $\tilde\gamma$ with identical combinatorial type (see Lemma 4.6 in~\cite{Staffa} and Lemma 2.6 in~\cite{LiokumovichStaffa}). In particular, $\gamma \cap U \neq \emptyset$ for $g$ sufficiently close to $\tilde g$. This shows $\mathcal{M}_U$ is open.

    Let $\tilde g$ be an arbitrary $C^k$ Riemannian metric on $M$ and let $\mathcal V$ be an arbitrary $C^k$--neighborhood of $\tilde g$. By the Bumpy Metrics Theorem for stationary geodesic nets (Theorem 0.3 in \cite{Staffa}) there exists $g\in\mathcal V$ such that every essential embedded stationary geodesic net for $g$ is nondegenerate. If any such net intersects $U$, then $g\in\mathcal M_U$ and we are done.
    
    Assume instead that every essential embedded stationary geodesic net for $g$ is contained in $M\setminus U$. By Theorem 1.1 in \cite{LiokumovichStaffa} we may also assume that the union of all embedded stationary geodesic nets in $(M,g)$ is dense. Hence there exists an embedded stationary geodesic net meeting $U$. By the assumption this net cannot be essential, so it must be a closed geodesic. Thus we get a closed geodesic $\alpha$ with $\alpha\cap U\neq\emptyset$.
    
    By density again there exists an embedded stationary geodesic net $\beta$ with $\beta\cap U\neq\emptyset$ and
    \[
    \operatorname{dist}(\alpha,\beta)<r_{\mathrm{inj}}(g),
    \]
    where $r_{\mathrm{inj}}(g)$ is the injectivity radius of $(M,g)$. By our hypothesis $\beta$ is also a closed geodesic.
    
    \textbf{Case 1.} Suppose $\alpha\cap\beta=\emptyset$. Pick points $a\in\alpha$ and $b\in\beta$ realizing the distance $\operatorname{dist}(\alpha,\beta)$ and set
    \[
    r:=\operatorname{dist}(a,b)>0.
    \]
    Let $\rho$ be the minimizing geodesic from $a$ to $b$; by minimality we have
    \begin{equation} \label{eq:scalar-product}
        \langle\dot\rho,\dot\alpha\rangle_a=\langle\dot\rho,\dot\beta\rangle_b=0.
    \end{equation}
    
    Fix \(t>0\) small and set
    \[
    p_{\pm}=\exp_a(\pm r v),\qquad a_t=\exp_a\bigl(t\dot\rho|_a\bigr),
    \]
    where \(v\) is the unit tangent to \(\alpha\) at \(a\). Assume \(\alpha\) is traversed \(p_-\to a\to p_+\) and $\rho$ is traversed \(b\to a_t\to a\). Consider geodesic segments \(\gamma_t^\pm\) starting at \(a_t\) whose initial velocities equal
    \[
    \pm r\,P_{a\to a_t}(v)+t\dot\rho,
    \]
    where \(P_{a\to a_t}\) denotes parallel transport along \(\rho\) from \(a\) to \(a_t\) (see Fig.~\ref{drawing_2}). Since parallel transport preserves inner products and by \eqref{eq:scalar-product} we have \(\langle\dot\rho,P_{a\to a_t}(v)\rangle=0\). Hence the two inward unit tangent vectors at \(a_t\) are
    \[
    \frac{1}{\sqrt{r^2+t^2}}\bigl(-rP_{a\to a_t}(v)-t\dot\rho\bigr),\qquad
    \frac{1}{\sqrt{r^2+t^2}}\bigl(rP_{a\to a_t}(v)-t\dot\rho\bigr),
    \]
    and their sum equals
    \begin{equation}\label{eq:vector}
    \frac{-2t}{\sqrt{r^2+t^2}}\,\dot\rho\Big|_{a_t}.
    \end{equation}

    \begin{figure}
        \centering
        \includegraphics[scale=0.6]{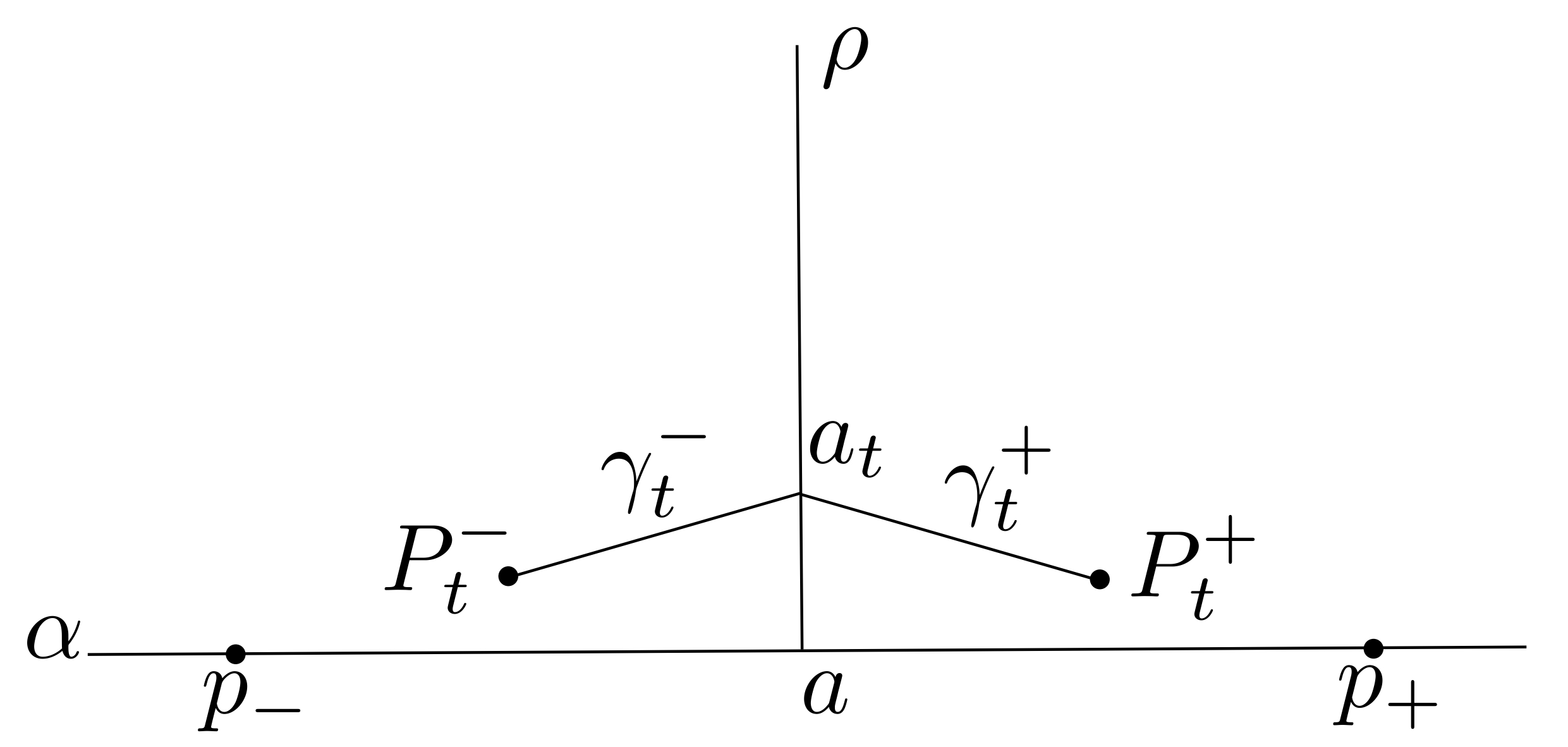}
        \caption{Geodesics $\alpha$ and $\rho$ near $a$}
        \label{drawing_2}
    \end{figure}
    
    We represent the geodesic segments \(\gamma_t^\pm\) as graphs over \(\alpha\).  Fix parameters \(s_1<s_2<s_3\) such that $\alpha(s_2) = a$ and $s_2 - s_1 = s_3 - s_2 = r/2$. Set
    \[
    \gamma_t^-(s)=\exp_{\alpha(s)}\bigl(v_t^-(s)\,\nu_t(s)\bigr),\qquad s\in[s_1,s_2],
    \]
    with \(\gamma_t^-(s_1)=P_t^-\) and \(\gamma_t^-(s_2)=a_t\), and
    \[
    \gamma_t^+(s)=\exp_{\alpha(s)}\bigl(v_t^+(s)\,\nu_t(s)\bigr),\qquad s\in[s_2,s_3],
    \]
    with \(\gamma_t^+(s_2)=a_t\) and \(\gamma_t^+(s_3)=P_t^+\). Here \(v_t^\pm\) are smooth scalar functions and \(\nu_t(s)\) is a smooth unit normal field along \(\alpha\). The points \(P_t^\pm\) are the unique points of \(\gamma_t^\pm\) whose nearest-point projections to \(\alpha\) are \(\alpha(s_1)\) and \(\alpha(s_3)\), respectively. Since
    \[
        a_t \rightarrow a
    \]
    and
    \[
        \frac{1}{\sqrt{r^2+t^2}}\bigl(\pm rP_{a\to a_t}(v)+t\dot\rho\bigr) = \frac{1}{\sqrt{1+(t/r)^2}}\bigl(\pm P_{a\to a_t}(v)+\frac{t}{r}\dot\rho\bigr) \rightarrow \pm v
    \]
    as $t \rightarrow 0$, smooth dependence of geodesics on initial conditions implies
    \[
        \lVert v_t^{\pm} \rVert_{C^{k+2}} \rightarrow 0
    \]
    as $t \rightarrow 0$.

    We now construct a smooth curve $\rho_t^-: [s_0, s_1] \rightarrow M$ such that
    \[
        \rho_t^-(s) = \exp_{\alpha(s)}\bigl(u_t^-(s) \nu_{t}(s)\bigr), \qquad \rho_t^-(s_0) = p_{-}, \quad \rho_t^-(s_1) = P_t^-,
    \]
    where $u_t^-: [s_0,s_1] \rightarrow \mathbb{R}$ is smooth. We require $\rho_t^-$ to agree with $\alpha$ at $p_{-}$ and with $\gamma_t^-$ at $P_t^-$ up to order $(k+2)$. We also require
    \[
        \lVert u_t^- \rVert_{C^{k+2}} = o(1)
    \]
    as $t \rightarrow 0$. Choose a smooth bump \(\psi\in C_c^\infty(\mathbb R)\) with \(\psi\equiv1\) on \([0,(s_1-s_0)/4]\) and \(\psi\equiv0\) outside of $[0,(s_1-s_0)/2]$. Define \(u_t^-\) on \([s_0,s_1]\) by taking a Taylor polynomial of \(v_t^-\) at \(s_1\) of degree \(k+2\) and multiplying it by a cutoff that forces the function and all its derivatives to vanish at \(s_0\). Explicitly,
    \[
    u_t^-(s)=\psi\bigl(s_1-s\bigr)\sum_{j=0}^{k+2}\frac{(s-s_1)^j}{j!}\bigl(v_t^-\bigr)^{(j)}(s_1).
    \]
    Then \(\rho_t^-(s)=\exp_{\alpha(s)}\bigl(u_t^-(s)\nu_t(s)\bigr)\) is smooth, satisfies the required conditions at the endpoints, and because \(\|v_t^-\|_{C^{k+2}}\to0\) we have
    \begin{equation}\label{eq:estimate-1}
        \|u_t^-\|_{C^{k+2}} \to 0\qquad (t\to0),
    \end{equation}
    so \(\rho_t^-\) is \(C^{k+2}\)--close to the corresponding arc of \(\alpha\). A similar construction produces a smooth curve \(\rho_t^+:[s_3,s_4]\to M\) of the form
    \[
    \rho_t^+(s)=\exp_{\alpha(s)}\bigl(u_t^+(s)\,\nu_t(s)\bigr),
    \]
    with \(\rho_t^+(s_3)=P_t^+\) and \(\rho_t^+(s_4)=p_+\). Here \(u_t^+\) is smooth on \([s_3,s_4]\), \(\rho_t^+\) matches \(\gamma_t^+\) at \(P_t^+\) and \(\alpha\) at \(p_+\) up to order \(k+2\), and
    \begin{equation}\label{eq:estimate-2}
        \|u_t^+\|_{C^{k+2}} \to 0\qquad (t\to0).
    \end{equation}

    \begin{figure}
        \centering
        \includegraphics[scale=0.6]{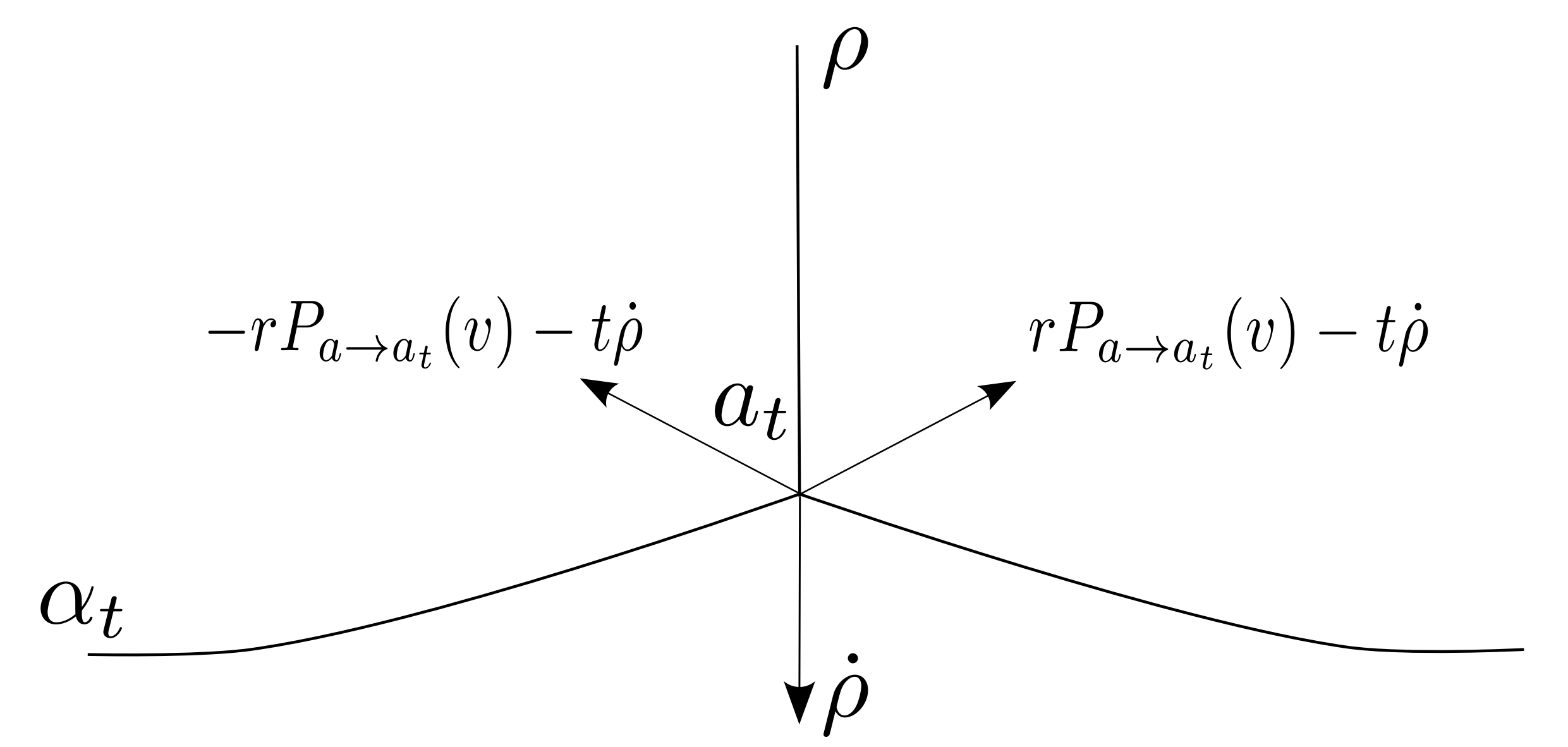}
        \caption{Triple junction at $a_t$}
        \label{drawing_3}
    \end{figure}
    
    Concatenating these pieces we obtain a closed loop \(\alpha_t\) smooth except at \(a_t\) (see Fig.~\ref{drawing_3}):
    \[
    \alpha_t=\rho_t^-\,\cup\,\gamma_t^-\,\cup\,\gamma_t^+\,\cup\,\rho_t^+\,\cup\,\alpha|_{[p_+,p_-]}.
    \]
    By~\eqref{eq:estimate-1} and~\eqref{eq:estimate-2} the curvature vector \(k_{\alpha_t}\) satisfies \(\|k_{\alpha_t}\|_{C^k}\to0\) as \(t\to0\).
    
    Since the rationals are dense, we may choose small \(t\) so that the scalar
    \[
    \lambda(t):=\frac{2t}{\sqrt{r^2+t^2}}
    \]
    is rational; write \(\lambda(t)=n(t)/m(t)\) with integers \(m(t),n(t)>0\). By~\eqref{eq:vector}, endowing the loop \(\alpha_t\) with multiplicity \(m(t)\) and the geodesic \(\rho\) with multiplicity \(n(t)\) makes the sum of inward unit tangent vectors at \(a_t\) equal to zero.
    
    We now perform a local conformal perturbation to make $\alpha_{t}$ geodesic. Let \(\{n_1,\dots,n_{n-1}\}\) be an orthonormal normal frame along \(\alpha_t\) and write the curvature vector as
    \[
    k_{\alpha_t}(s)=(0,k^1_{\alpha_t}(s),\dots,k^{n-1}_{\alpha_t}(s))
    \]
    in Fermi coordinates \((s,h_1,\dots,h_{n-1})\) along \(\alpha_t\). Fix a cutoff \(\chi_{r/4}\) supported in the radius \(r/4\) normal tube around \(\alpha_t\) and define
    \[
    f_t(s,h_1,\dots,h_{n-1})=\chi_{r/4}(h_1,\dots,h_{n-1})\sum_{i=1}^{n-1} h_i\,k^i_{\alpha_t}(s).
    \]
    For the conformal metric \(g_t=e^{2f_t}g\) the geodesic curvature transforms as
    \[
    k_{\alpha_t,f_t}=e^{-f_t}\bigl(k_{\alpha_t}-(\nabla_g f_t)^\perp\bigr),
    \]
    where \((\nabla_g f_t)^\perp\) denotes the projection of \(\nabla_g f_t\) to the normal bundle of \(\alpha_t\). By construction \((\nabla_g f_t)^\perp=k_{\alpha_t}\) on the support of \(k_{\alpha_t}\), so \(k_{\alpha_t,f_t}\equiv 0\). Thus \(\alpha_t\) becomes a geodesic of \((M,g_t)\).
    
    Because \(\|k_{\alpha_t}\|_{C^k}\to0\) and the cutoff is fixed, we have \(\|f_t\|_{C^k}\to0\) and hence \(\|g_t-g\|_{C^k}\to0\) as $t\to0$. Performing the same localized perturbation in a disjoint neighborhood of \(\beta\) for sufficiently small \(t\) produces a metric \(g_t^*\in\mathcal V\) and a stationary geodesic net
    \[
    \Gamma_t = m(t)\,\alpha_t \cup n(t)\,\rho|_{b_t \to a_t} \cup m(t)\,\beta_t.
    \]
    The stationary geodesic net \(\Gamma_t\) is essential and embedded, and it can be made nondegenerate by an arbitrarily small further perturbation using the Bumpy Metrics Theorem for geodesic nets.

    \begin{figure}
        \centering
        \includegraphics[scale=0.5]{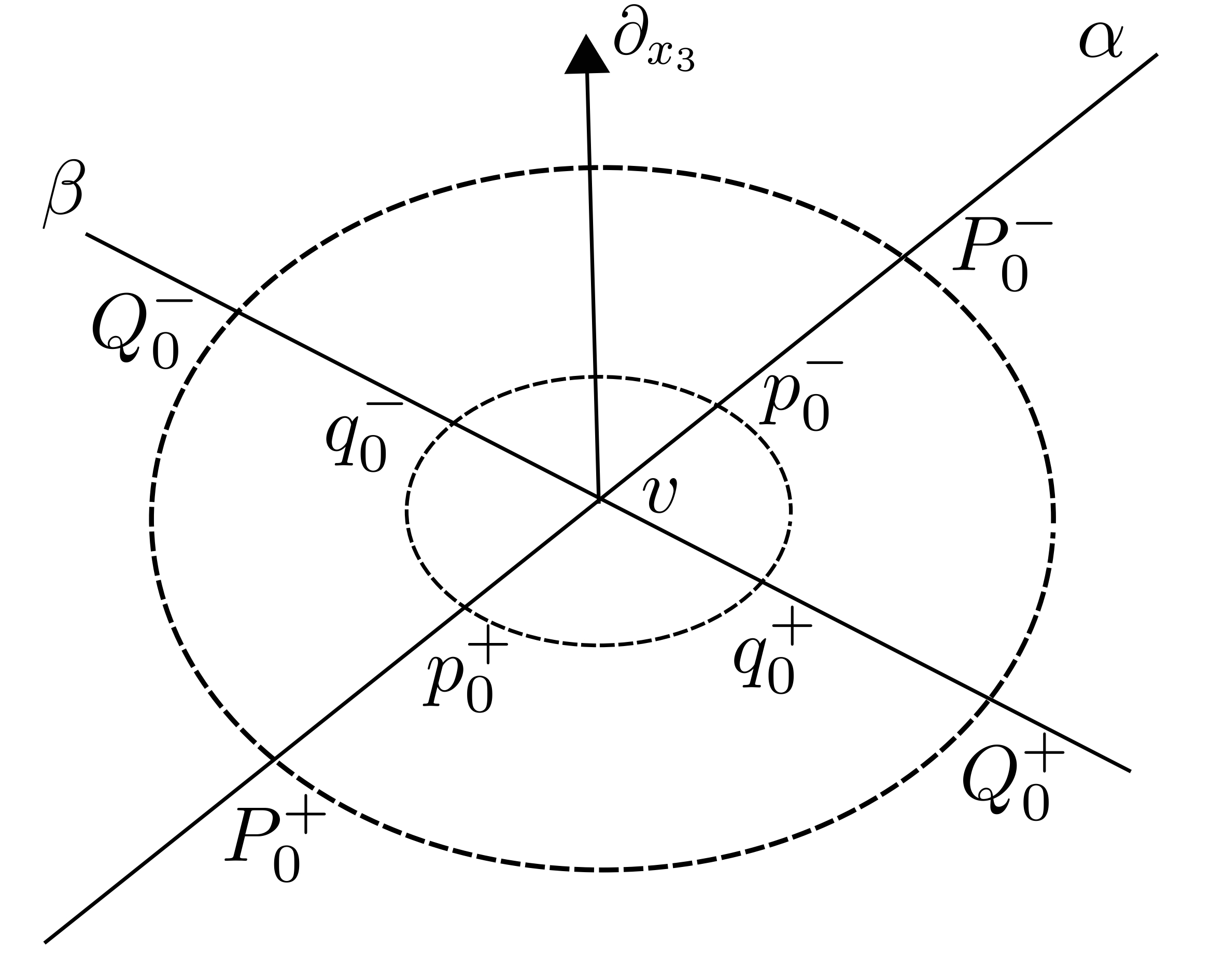}
        \caption{Geodesics $\alpha$ and $\beta$ near $v$}
        \label{drawing_4}
    \end{figure}
    
    \medskip\noindent\textbf{Case 2.} Suppose \(v\in\alpha\cap\beta\). Let \(r_{\mathrm{inj}}\) be the injectivity radius of \((M,g)\). Choose \(r<r_{\mathrm{inj}}\) and let \(V=\exp_v(B_r(0))\) be the geodesic neighborhood of \(v\) of radius \(r\). Define the exponential chart \(\tilde\varphi\colon B_r(0)\subset\mathbb R^n\to V\) by \(\tilde\varphi(x)=\exp_v(x)\). Introduce rescaled coordinates \(\varphi\colon B_{3/2}(0)\to V\) by
    \[
    \varphi(x)=\tilde\varphi\!\Big(\tfrac{2r}{3}x\Big)=\exp_v\!\Big(\tfrac{2r}{3}x\Big),\qquad x\in B_{3/2}(0).
    \]
    Consider the rescaled metric \(\tfrac{9}{4r^2}\varphi^*g\) on \(B_{3/2}(0)\). Identify \(\alpha\) and \(\beta\) with their preimages under \(\varphi\), which are straight lines in this chart (see Fig.~\ref{drawing_4}). Rotate the chart so that
    \(\operatorname{span}\{\partial_{x_1},\partial_{x_2}\}=\operatorname{span}\{\dot\alpha(v),\dot\beta(v)\}\).
    Denote
    \[
    \alpha\cap\partial B_{1/2}(0)=\{p_0^+,p_0^-\},\qquad
    \alpha\cap\partial B_1(0)=\{P_0^+,P_0^-\},
    \]
    and
    \[
    \beta\cap\partial B_{1/2}(0)=\{q_0^+,q_0^-\},\qquad
    \beta\cap\partial B_1(0)=\{Q_0^+,Q_0^-\}.
    \]
    
    Assume \(\alpha\) is traversed \(P_0^-\to v\to P_0^+\) and \(\beta\) is traversed \(Q_0^-\to v\to Q_0^+\). For small \(t>0\) set
    \[
    \omega_t^+:=\dot\alpha(v)+t\,\partial_{x_3},\qquad \omega_t^-:=-\dot\alpha(v)+t\,\partial_{x_3},
    \]
    and let \(\sigma_t^+\) and \(\sigma_t^-\) be the geodesics starting at \(v\) with initial velocities \(\omega_t^+\) and \(\omega_t^-\), respectively. Since \(\operatorname{diam}(V)<r_{\mathrm{inj}}\), \(\sigma_t^\pm\) meet \(\alpha\) and \(\beta\) only at \(v\) inside \(V\). Let \(p_t^\pm\) denote the unique points of \(\sigma_t^\pm\) whose nearest-point projections to \(\alpha\) are \(p_0^\pm\). Fix \(s_0\) with \(\alpha(s_0)=v\). Write \(\sigma_t^+:[s_0,s_0+1/2]\to M\) and \(\sigma_t^-:[s_0-1/2,s_0]\to M\) as graphs over \(\alpha\):
    \[
    \sigma_t^\pm(s)=\exp_{\alpha(s)}\bigl(v_t^\pm(s)\,\nu_t(s)\bigr),
    \]
    with
    \[
    \sigma_t^+(s_0)=v,\quad \sigma_t^+(s_0+1/2)=p_t^+,
    \qquad
    \sigma_t^-(s_0-1/2)=p_t^-,\quad \sigma_t^-(s_0)=v,
    \]
    where \(v_t^\pm\) are smooth on \([s_0-1/2,s_0+1/2]\) and \(\nu_t(s)\) is a smooth unit normal field along \(\alpha\). Smooth dependence of geodesics on initial conditions implies
    \[
    \|v_t^\pm\|_{C^{k+2}}\to0\qquad(t\to0).
    \]
    
    We construct smooth connectors \(\rho_t^+:[s_0+1/2,s_0+1]\to M\) and \(\rho_t^-:[s_0-1,s_0-1/2]\to M\) as follows. Choose a smooth bump \(\psi\in C_c^\infty(\mathbb R)\) with \(\psi\equiv1\) near \(0\) and supported in a small interval. Define
    \[
    u_t^+(s)=\psi(s-s_0-1/2)\sum_{j=0}^{k+2}\frac{(s-s_0-1/2)^j}{j!}\,(v_t^+)^{(j)}(s_0+1/2)
    \qquad(s\in[s_0+1/2,s_0+1]),
    \]
    and
    \[
    u_t^-(s)=\psi(s-s_0+1/2)\sum_{j=0}^{k+2}\frac{(s-s_0+1/2)^j}{j!}\,(v_t^-)^{(j)}(s_0-1/2)
    \qquad(s\in[s_0-1,s_0-1/2]).
    \]
    Set
    \[
    \rho_t^+(s)=\exp_{\alpha(s)}\bigl(u_t^+(s)\,\nu_t(s)\bigr),\qquad
    \rho_t^+(s_0+1/2)=p_t^+,\quad \rho_t^+(s_0+1)=P_0^+,
    \]
    and
    \[
    \rho_t^-(s)=\exp_{\alpha(s)}\bigl(u_t^-(s)\,\nu_t(s)\bigr),\qquad
    \rho_t^-(s_0-1)=P_0^-,\quad \rho_t^-(s_0-1/2)=p_t^-.
    \]
    The functions \(u_t^\pm\) match \(v_t^\pm\) at the endpoints up to order \(k+2\) and satisfy
    \begin{equation} \label{eq:estimate-3}
       \|u_t^\pm\|_{C^{k+2}} \to 0\qquad(t\to0).
    \end{equation}

    \begin{figure}
        \centering
        \includegraphics[scale=0.5]{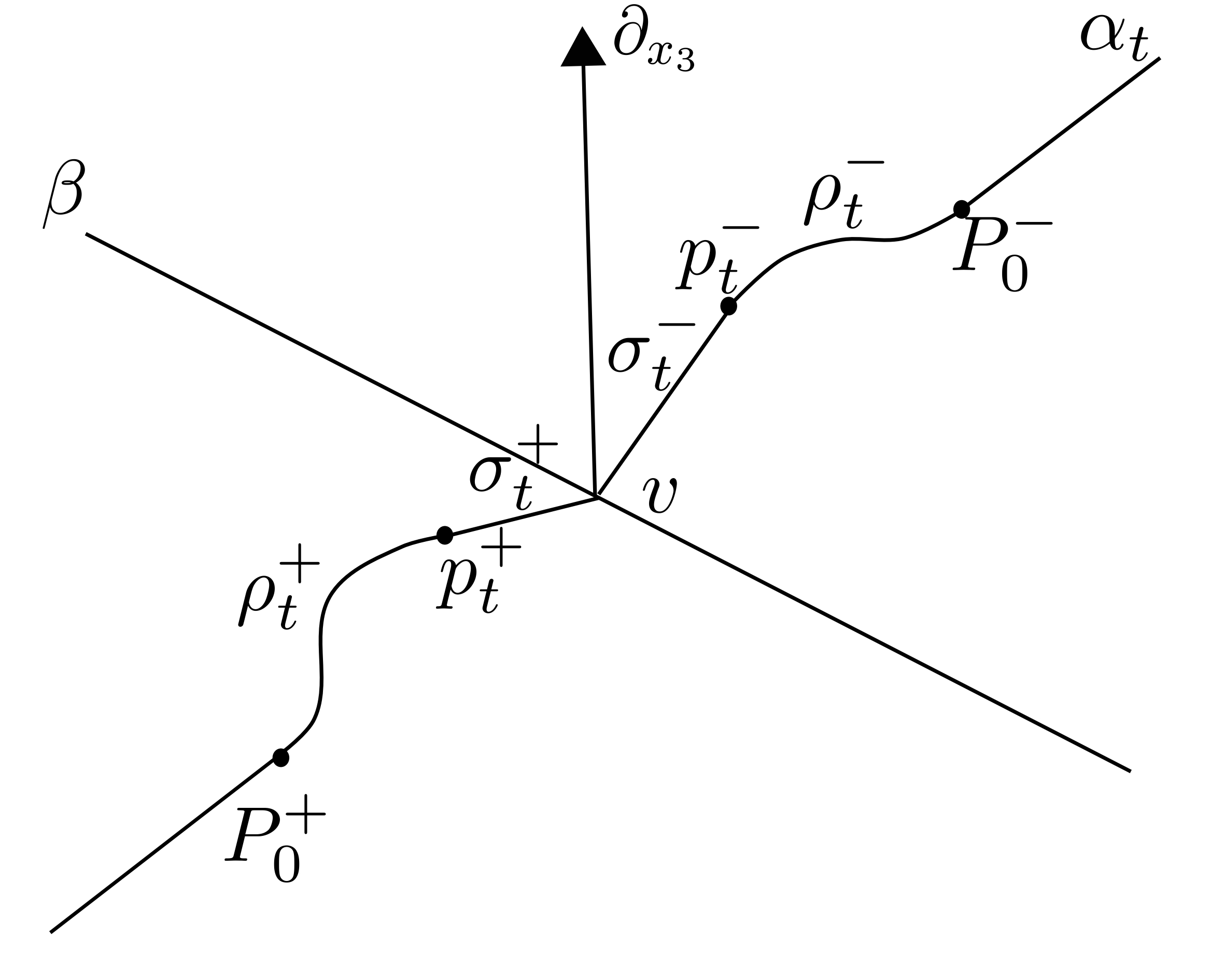}
        \caption{Perturbation of $\alpha$}
        \label{drawing_5}
    \end{figure}
    
    Concatenating the pieces yields a closed loop \(\alpha_t\) smooth except \(v\) (see Fig.~\ref{drawing_5}):
    \[
    \alpha_t=\sigma_t^+\ \cup\ \rho_t^+\ \cup\ \alpha|_{[P_0^+,P_0^-]}\ \cup\ \rho_t^-\ \cup\ \sigma_t^-.
    \]
    
    We now perform a local conformal perturbation to make \(\alpha_t\) geodesic. Let \(\{n_1,\dots,n_{n-1}\}\) be an orthonormal normal frame along \(\alpha_t\) and write the curvature vector in Fermi coordinates \((s,h_1,\dots,h_{n-1})\) as
    \[
    k_{\alpha_t}(s)=(0,k^1_{\alpha_t}(s),\dots,k^{n-1}_{\alpha_t}(s)).
    \]
    Set
    \[
    d:=\operatorname{dist}\bigl(\alpha|_{[P_0^-,p_0^-]}\cup\alpha|_{[p_0^+,P_0^+]},\ \beta|_{[Q_0^-,q_0^-]}\cup\beta|_{[q_0^+,Q_0^+]}\bigr)>0.
    \]
    Fix a cutoff \(\chi_{d/4}\) supported in the radius \(d/4\) normal tube around \(\alpha_t\) and define
    \[
    f_t(s,h_1,\dots,h_{n-1})=\chi_{d/4}(h_1,\dots,h_{n-1})\sum_{i=1}^{n-1} h_i\,k^i_{\alpha_t}(s).
    \]
    For the conformal metric \(g_t=e^{2f_t}g\) the geodesic curvature transforms as
    \[
    k_{\alpha_t,f_t}=e^{-f_t^+}\bigl(k_{\alpha_t}-(\nabla_g f_t^+)^\perp\bigr).
    \]
    By construction \((\nabla_g f_t)^\perp=k_{\alpha_t}\) on the support of \(k_{\alpha_t}\), hence \(k_{\alpha_t,f_t}\equiv0\). Thus \(\alpha_t\) becomes a geodesic of \((M,g_t)\).

    \begin{figure}
        \centering
        \includegraphics[scale=0.5]{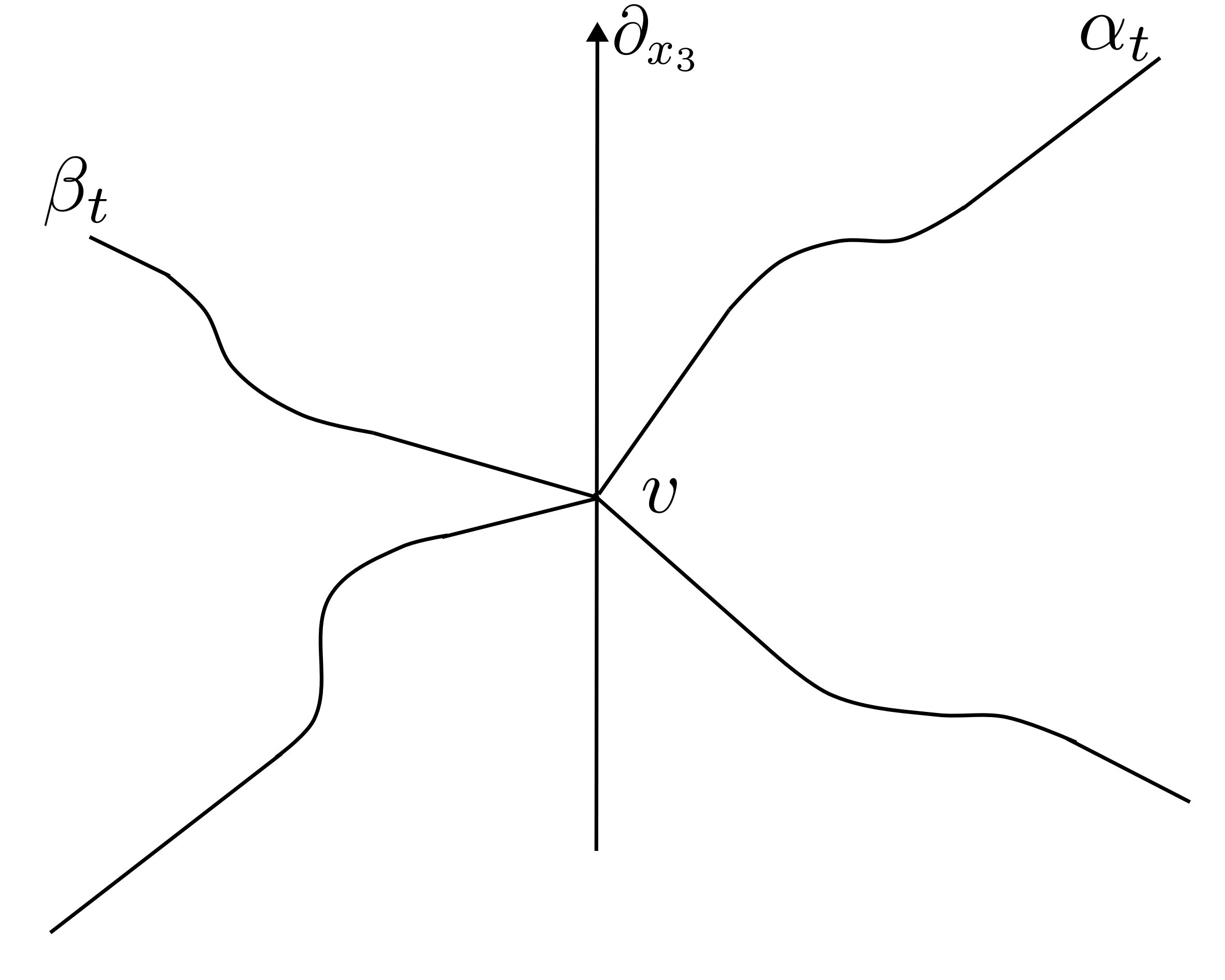}
        \caption{Stationary geodesic net $\Gamma_t = \alpha_t \cup \beta_t$}
        \label{drawing_6}
    \end{figure}
    
    Perform an analogous construction for \(\beta\). Set
    \[
    \tau_t^+:=\dot\beta(v)-t\partial_{x_3},\qquad \tau_t^-:=-\dot\beta(v)-t\partial_{x_3},
    \]
    let \(\eta_t^\pm\) be the geodesics starting at \(v\) with initial velocities \(\tau_t^\pm\), and construct connectors \(\lambda_t^\pm\) as above. Define
    \[
    \beta_t=\eta_t^+\ \cup\ \lambda_t^+\ \cup\ \beta|_{[Q_0^+,Q_0^-]}\ \cup\ \lambda_t^-\ \cup\ \eta_t^-,
    \]
    and set
    \[
    \tilde f_t(s,h_1,\dots,h_{n-1})=\chi_{d/4}(h_1,\dots,h_{n-1})\sum_{i=1}^{n-1} h_i\,k^i_{\beta_t}(s).
    \]
    By the choice of \(d\) we have \(\operatorname{supp}(f_t)\cap\operatorname{supp}(\tilde f_t)=\emptyset\).
    
    By~\eqref{eq:estimate-3} we have \(\|k_{\alpha_t}\|_{C^k},\|k_{\beta_t}\|_{C^k}\to0\) as \(t\to0\). Hence, \(\|f_t\|_{C^k} + \|\tilde f_t\|_{C^k}\to0\). Let \(g_t^*=e^{2(f_t+\tilde f_t)}g\). Then \(\|g_t^*-g\|_{C^k}\to0\) as \(t\to0\).
    
    By construction \(\alpha_t\) and \(\beta_t\) are geodesics for \(g_t^*\) and meet at \(v\) with inward unit tangents summing to zero. Therefore
    \[
    \Gamma_t:=\alpha_t\cup\beta_t
    \]
    is an essential, embedded stationary geodesic net for \((M,g_t^*)\) intersecting \(U\) (see Fig.~\ref{drawing_6}). For sufficiently small \(t\) we have \(g_t^*\in\mathcal V\). By the Bumpy Metrics Theorem for geodesic nets a generic arbitrarily small further perturbation makes \(\Gamma_t\) nondegenerate. This completes the proof.
\end{proof}

\begin{proof}[Proof of Main Theorem]
    Let $k \in \mathbb{Z}_{> 0}$ and $\{U_i\}$ be a countable basis of $M$. By Proposition~\ref{prop}, each $\mathcal{M}^k_{U_i}$ is open and dense in $\mathcal{M}^k$, so $\bigcap_i\mathcal{M}^k_{U_i}$ is $C^k$-Baire generic in $\mathcal{M}^k$. By Lemma 6.2 in~\cite{Staffa}, the set $\bigcap_i \mathcal{M}^{\infty}_{U_i} = \bigcap_k \bigcap_i \mathcal{M}^k_{U_i}$ is $C^\infty$-Baire generic in $\mathcal{M}^{\infty}$.
\end{proof}

\vspace{0.5cm} % Adjust vertical space as needed
\noindent Department of Mathematics, University of Toronto, Toronto, Canada\\
\textit{E-mail address}: \texttt{talant.talipov@mail.utoronto.ca}

\end{document}